\documentclass[11pt,fleqn]{article}
\usepackage{amsmath,epic,amssymb}
\usepackage{amsthm,amsfonts,amscd,epsfig}
\usepackage[dvipsnames]{xcolor}
\usepackage{authblk}
\usepackage{pgfplots}

\newtheorem{theorem}{Theorem}[section]
\newtheorem{lemma}[theorem]{Lemma}
\newtheorem{remark}[theorem]{Remark}
\newtheorem{proposition}[theorem]{Proposition}
\newtheorem{corollary}[theorem]{Corollary}
\newtheorem{definition}[theorem]{Definition}
\newtheorem{example}[theorem]{Example}
\newcount\refno
\newcommand\be{\begin{equation}}
\newcommand\ee{\end{equation}}
\newcommand\bn{\begin{eqnarray}}
\newcommand\en{\end{eqnarray}}
\newcommand\bns{\begin{eqnarray*}}
\newcommand\ens{\end{eqnarray*}}
\newcommand\bd{\begin{definition}}
\newcommand\ed{\end{definition}}
\newcommand\br{\begin{remark}}
\newcommand\er{\end{remark}}
\newcommand\bt{\begin{theorem}}
\newcommand\et{\end{theorem}}
\newcommand\bp{\begin{proposition}}
\newcommand\ep{\end{proposition}}
\newcommand\bc{\begin{corollary}}
\newcommand\ec{\end{corollary}}
\newcommand\bl{\begin{lemma}}
\newcommand\el{\end{lemma}}
\newcommand\pf{\begin{proof}}

\newcommand\bN{{\mathbb N}}

\newcommand{\F}{\mbox{$\mathcal{F}$}}

\makeindex

\allowdisplaybreaks

\begin{document}

\title{Total Positivity of Quasi-Riordan Arrays}

\author[1]{Tian-Xiao He}
\author[2]{Roksana S\l{}owik}
\affil[1]{Department of Mathematics, Illinois Wesleyan University, Bloomington, IL 61702-2900, Indiana, USA}
\affil[2]{Faculty of Applied Mathematics, Silesian University of Technology, Kaszubska 23, 44-100 Gliwice, Poland}

\pagestyle{myheadings} 
\markboth{T. X. He and R. S\l{}owik}{TP Quasi-Riordan Arrays}

\maketitle

\begin{abstract}
\noindent 
In this paper the total positivity of quasi-Riordan arrays is investigated with use of the sequence characterization of quasi-Riordan arrays. Due to the correlation between quasi-Riordan arrays and Riordan arrays, this study is an in-depth discussion of the total positivity of Riordan arrays.


\end{abstract}

\section{Introduction}\label{Sec1}

Following Karlin \cite{Kar} and Pinkus \cite{Pin}, an infinite matrix is called totally positive (abbreviately, TP), if its minors of all orders are nonnegative. An infinite sequence of nonnegative real numbers  $(a_n)_{n\geq 0}$ is called a P\'olya frequency sequence (abbreviately, PF), if its Toeplitz matrix

\[
\left[a_{i-j}\right]_{i,j\geq 0}=\left[ \begin{array} {lllll} a_0 & & & &  \\
a_1& a_0 & & & \\ a_2 &a_1& a_0& &  \\ a_3& a_2 & a_1& a_0 & \\
\vdots& \vdots &\vdots&\vdots & \ddots \end{array}\right]
\]

\noindent
is TP. We say that a finite sequence $(a_0, a_1,\ldots, a_n)$ is PF if the corresponding infinite sequence $(a_0, a_1, \ldots, a_n, 0, \ldots)$ is PF. Denote by ${\bN}$ the set of all nonnegative integers. A fundamental characterization for PF sequences is given by Schoenberg et al.\cite{AESW, ASW} (see also Karlin \cite{Kar}), which states that a sequence $(a_n)_{n\geq 0}$ is PF if and only if its generating function can be written as 

\begin{equation}\label{0}
\sum_{n\geq 0} a_n t^n=C t^ke^{\gamma t}\frac{\Pi_{j\geq 0} (1+\alpha_j t)}{\Pi_{j\geq 0} (1-\beta_j t)},
\end{equation}
where $C>0$, $k\in {\bN}$, $\alpha_j$, $\beta_j$, $\gamma \geq 0$, and $\sum_{j\geq 0}(\alpha_j+\beta_j)<\infty$. In this case, the above generating function is called a P\'olya frequency formal power series. For more preliminary materials and more relevant results see, for example, Brenti \cite{Bre}, Mao, Mu, and Wang \cite{MMW} and Pinkus \cite{Pin}.

Let $A$ and $B$ be $m\times m$ and $n\times n$ matrices, respectively. Then we define the direct sum
of $A$ and $B$ by

\begin{equation}\label{2.5} 
A\oplus B =\left[ \begin{matrix} A &0 \\ 0 &B\end{matrix}\right]_{(m+n)\times(m+n)}.
\end{equation}

Let us recall some basic definitions and introduce the notation.

First, the Riordan arrays (or Riordan matrices) are infinite, lower triangular matrices defined by the generating function of their columns. With the matrix multiplication, the set of all Riordan arrays form a group, called {\em the Riordan group} (see Shapiro, Getu, Woan and Woodson \cite{SGWW}). 

More formally, let us consider the set of formal power series ring ${\mathcal F} = {{\mathbb K}}[\![$$t$$]\!]$, where ${{\mathbb K}}$ is the field ${{\mathbb R}}$ or ${{\mathbb C}}$. The \emph{order} of $f(t)  \in {\mathcal F}$, $f(t) =\sum_{k=0}^\infty f_kt^k$ ($f_k\in {{\mathbb K}}$), is the minimum number $r\in{{\mathbb N}}$ such that $f_r \neq 0$, where ${{\mathbb N}}=\{0,1,2\ldots\}$ is the set of all natural numbers. We denote by ${\mathcal F}_r$ the set of formal power series of order $r$. Let $g(t) \in {\mathcal F}_0$ and $f(t) \in {\mathcal F}_1$; the pair $(g,\,f )$ defines the {\em (proper) Riordan array} $R=(d_{n,k})_{n,k\in {{\mathbb N}}}=(g, f)$ having
  
\begin{equation}\label{Radef}
d_{n,k} = [t^n]g(t) f(t)^k
\end{equation}
or, in other words, having $g f^k$ as the generating function of the $k$th column of $(g,f)$.  Hence,
\begin{equation}\label{2.6-0}
(g,f):=(g,gf,gf^2,gf^3,\ldots),
\end{equation}
where $g$, $gf$, $gf^2$, $gf^3\cdots$ are the generating functions of the $0$th, $1$st, $2$nd, $3$rd, $\cdots$, columns of the matrix $(g,f)$, respectively. The {\it first fundamental theorem of Riordan arrays} can be represented as 

\[
(g(t), f(t)) h(t)=g(t) (h\circ f)(t),
\]
which can also be simplified to $(g,f)h=gh(f)$. Thus we immediately see that the usual row-by-column product of two Riordan arrays is also a Riordan array:
\begin{equation}\label{Proddef}
    (g_1,\,f_1 )  (g_2,\,f_2 ) = (g_1 g_2(f_1),\,f_2(f_1)).
\end{equation}
The Riordan array $I = (1,\,t)$ is the identity matrix because its entries are $d_{n,k} = [t^n]t^k=\delta_{n,k}$.

Let $(g\left( t\right),\,f(t)) $ be a Riordan array. Then its inverse is

\begin{equation}
(g\left( t\right),\,f(t))^{-1}=\left( \frac{1}{g(\overline{{f}}(t))},
\overline{f}(t),\right)  \label{Invdef}
\end{equation}%
where $\overline {f}(t)$ is the compositional inverse of $f(t)$, i.e., $(f\circ 
\overline{f})(t)=(\overline{f}\circ f)(t)=t$. In this way, the set $\mathcal{
R}$ of all proper Riordan arrays forms a group (see \cite{SGWW}) called the Riordan group.

Here is a list of six important subgroups of the Riordan group (see \cite{SGWW, Sha, Barry}).

\begin{itemize}
\item the {\it Appell subgroup} $\{ (g(t),\,t):g\in {\mathcal F}_0\}$.
\item the {\it Lagrange (associated) subgroup} $\{(1,\,f(t):  f\in{\mathcal F}_1)\}$.
\item the {\it $k$-Bell subgroup} $\{(g(t),\, t(g(t))^k):g\in {\mathcal F}_0\}$, where $k$ is a fixed positive integer. 
\item the {\it hitting-time subgroup} $\{(tf'(t)/f(t),\, f(t)): f\in {\mathcal F}_1\}$.
\item the {\it derivative subgroup} $\{ (f'(t), \, f(t)):f\in {\mathcal F}_1\}$.
\item the {\it checkerboard subgroup} $\{ (g(t),\, f(t)):g\in {\mathcal F}_0, f\in {\mathcal F}_1\},$ where $g$ is an even function and $f$ is an odd function. 
\end{itemize}

The $1$-Bell subgroup is referred to as the Bell subgroup for short, and the Appell subgroup can be considered as the $0$-Bell subgroup if we allow $k=0$ to be included in the definition of the $k$-Bell subgroup.

An infinite lower triangular matrix $[d_{n,k}]_{n,k\in{{{\mathbb N}}}}$ is a Riordan array if and only if a unique sequence $A=(a_0\not= 0, a_1, a_2,\ldots)$ exists such that for every $n,k\in{{{\mathbb N}}}$  
\begin{equation}\label{eq:1.1}
d_{n+1,k+1} =a_0 d_{n,k}+a_1d_{n,k+1}+\cdots +a_nd_{n,n}. 
\end{equation} 
This is equivalent to 
\begin{equation}\label{eq:1.2}
f(t)=tA(f(t))\quad \text{or}\quad t=\bar f(t) A(t).
\end{equation}
Here, $A(t)$ is the generating function of the $A$-sequence. The first formula of \eqref{eq:1.2} is also called the {\it second fundamental theorem of Riordan arrays}.

Moreover, there exists a unique sequence $Z=(z_0, z_1,z_2,\ldots)$ such that every element in column $0$ can be expressed as the linear combination 
\begin{equation}\label{eq:1.3}
d_{n+1,0}=z_0 d_{n,0}+z_1d_{n,1}+\cdots +z_n d_{n,n},
\end{equation}
or equivalently,
\begin{equation}\label{eq:1.4}
g(t)=\frac{1}{1-tZ(f(t))},
\end{equation}
in which and thoroughly we always assume $g(0)=g_0=1$, a usual hypothesis for proper Riordan arrays. From \eqref{eq:1.4}, we may obtain 

\[
Z(t)=\frac{g(\bar f(t))-1}{\bar f(t)g(\bar f(t))}.
\]

$A$- and $Z$-sequence characterizations of Riordan arrays were introduced, developed, and/or studied in Merlini, Rogers, Sprugnoli, and Verri \cite{MRSV}, Roger \cite{Rog}, Sprugnoli and the author \cite{HS}, Jean-Louis and Nkwanta \cite{JLN}, Luz\'on, Mor\'on, Prieto-Martinez \cite{LMPM17, LMPM}, and \cite{He15} as well as their references.

Secondly, we recall the definition of quasi-Riordan arrays given in \cite{He}.

\begin{definition}\label{def:2.3}\cite{He} 
Let $g\in \F_0$ with $g(0)=1$ and $f\in \F_1$. We call the following matrix a quasi-Riordan array and denote it by $[g,f]$.
\begin{equation}\label{2.6}
[g,f]:=(g,f,tf,t^2f,\ldots),
\end{equation}
where $g$, $f$, $tf$, $t^2f\cdots$ are the generating functions of the $0$th, $1$st, $2$nd, $3$rd, $\cdots$, columns of the matrix $[g,f]$, respectively. It is clear that $[g,f]$ can be written as 
\begin{equation}\label{2.8}
[g,f]=\left( \begin{matrix} g(0) & 0\\ (g-g(0))/t & (f/t,t)\end{matrix}\right),
\end{equation}
where $(f,t)=(f,tf,t^2f,t^3f,\ldots)$ and $(f/t, t)$ is an Appell Riordan array. Particularly, if $f=tg$, then the quasi-Riordan array $[g,tg]=(g, t)$,the Appell-type Riordan array. 
\end{definition}

The relationship between the Riordan arrays and quasi-Riordan arrays can be presented in the following theorem.

\begin{theorem}\label{thm:2.4}\cite{He}
Let $(g,f)=(d_{n,k})_{n,k\geq 0}$ be a Riordan array, and let $([1]\oplus (g,f))$ and $[g,f]$ be defined by \eqref{2.5} and \eqref{2.6}, respectively. Then $(g,f)$ has the expression
\begin{equation}\label{2.9}
(g,f)=[g,f]([1]\oplus (g,f)).
\end{equation}
\end{theorem}

For an integer $r\geq 0$, denote the $r$-th truncation of a formal power series $h=\sum_{n\geq 0} h_n t^n$ by $h|_r:=\sum^r_{n=0} h_nt^n$. For an integer $n\geq 0$, denote by $(g, f)_n$ the $n$-th order leading principle submatrix of the Riordan array $(g(t),f(t))_n$. This notation is introduced by Luz\'on et al in \cite{LMMPMS} and to the one used later by Luz\'on and her collaborators. Similarly, $[g,f]_n$ denotes the $n$-th order leading principle submatrix of the quasi-Riordan array $[g,f]$ defined by \eqref{2.6}, namely 
\begin{equation}\label{2.12}
[g,f]_n=\left( \begin{matrix} g(0) & 0\\ ((g-g(0))/t)|_{n-1} & (f,t)_{n-1}\end{matrix}\right),
\end{equation}
where the $n-1$-st truncation of $(g-g(0)/t)|_{n-1}$ is 
\[(g-g(0)/t)|_{n-1}=\sum^{n-1}_{k=1}g_kt^{k-1},\] 
and $(f,t)_{n-1}=(f, tf, t^2f, \ldots, t^{n-1} f).$ We call $[g,f]_n$ the $n$-th truncation of the quasi-Riordan array $[g,f]$.

\begin{corollary}\label{cor:2.5}\cite{He, MMW}
Let $(g,f)=(d_{n,k})_{n,k\geq 0}$ be a Riordan array with $g=\sum_{n\geq 0} g_n t^n$ and $f=\sum_{n\geq 1} f_n t^n$, and let $(g,f)_n$ be the $n$-th order leading principle submatrix of $(g,f)$. Then we have the recurrence relation of $(g,f)_n$ in the following form:
\begin{equation}\label{2.11}
(g,f)_n=[g,f]_n([1]\oplus (g,f)_{n-1}).
\end{equation}
\end{corollary}

Let $R=(g(t), f(t))$ be a Riordan array, where $g(t)=\sum_{n\geq 0} g_nt^n$ and $f(t)=\sum_{n\geq 1}f_nt^n$. If the lower triangular matrix 

\begin{equation}\label{eq:1}
Q=\left [ \begin{array}{llllll} g_0& 0& 0& 0& 0& \cdots\\
g_1& f_1 & 0& 0& 0&\cdots\\
g_2 &f_2& f_1& 0& 0& \cdots\\
g_3& f_3& f_2& f_1&0&\cdots\\
\vdots &\vdots& \vdots& \vdots&\vdots&\ddots\end{array}\right]=\left( g(t), f(t), tf(t), t^{2}f(t),\ldots\right)
\end{equation}
is totally positive (TP), then so is $R$ \cite{MMW}. This can be proved easily by using Corollary \ref{cor:2.5} and mathematical induction. 

The question is when $Q$ is totally positive. It is obvious that the P\'olya frequency property for the formal power series $g$, $f$ is not sufficient to ensure TP of $Q$.

\begin{example}
\label{ex:gf_not_TP}
Let $g(t)=(1+t)^2$ and $f(t)=t/(1-t)$. Then both $g(t)$ and $f(t)$ are P\'olya frequency, and $(g, f)$ is TP. However, the minor of $[g,f]$ 

\[
M^{1,2,3}_{0,1,2}=\det\left( \begin{array}{lll} 2&1&0\\1& 1&1\\0&1&1\end{array}\right)=-1<0.
\]
Hence, $[g,f]$ is not TP although $(g,f)$ is TP. 
\end{example}

In this paper, we search for the conditions of TP quasi-Riordan arrays and the sequence characterization of TP quasi-Riordan arrays. From the last example, we know that the formal power series $g(t)$ plays an important role in identifying whether a quasi-Riordan array is TP. Hence, in the next section, we discuss the conditions and their equivalents for TP quasi-Riordan arrays $[g,f]$ with some special power series $g$. In Section \ref{Sec3}, we study the sequence characterization of TP quasi-Riordan arrays. In particular, we present some generalizations of results given in \cite{CLW}. More specifically, Theorems \ref{pro:general} and \ref{prop:minors} in this section are generalizations of Theorem 2.1 and Theorem 2.3 of \cite{CLW} for quasi-Riordan arrays. In Section \ref{Sec4} we will present some more examples with various $g$'s for TP and non-TP quasi-Riordan arrays $[g,f]$, respectively. 
From a list of examples showing in this section and the results in the previous sections, one may see a kind of independence of a TP quasi-Riordan array $[g,f]$ and the properties of that its formal power series $g$ and $f$ are the P\'olya frequency formal power series. 

\section{TP Quasi-Riordan Arrays}
\label{Sec2}

Let us consider two special types of Riordan arrays and discuss their TP property and the TP property for the corresponding quasi-Riordan array. 

If $g(t)=1$, then from \eqref{2.8} we have 

\begin{equation}\label{2.0}
[g,f]=\left( \begin{matrix} 1 & 0\\ 0& (f/t,t)\end{matrix}\right). 
\end{equation}
Hence, 

\begin{proposition}
\label{pro:Toeplitz_case} 
Let $f\in\mathcal F_1$.
The following conditions are equivalent.
\begin{enumerate}
\item $f$ is a P\'olya frequency sequence.
\item The Riordan array $(f/t,t)$ is TP.
\item The Riordan array $(1,f)$ is TP.
\item The quasi-Riordan array $\left[1,\frac{f(t)}t\right]$ is TP.
\end{enumerate}
\end{proposition}

\begin{proof} Since $(f/t,t)$ is a Riordan array, from the definition of the TP matrices, we immediately obtain Proposition 
\ref{pro:Toeplitz_case}.
\end{proof}

Consider the case when $g(t)=1+g_1t$ and $f(t)=f_1t+f_2t^2$. Then 

\[
[g,f]=\left[\begin{array}{llll} 1& 0 & & \\ g_1 & f_1 &0 & \\ & f_2& f_1 & \\ & & f_2 &\\  &   \ddots &\ddots &\ddots\end{array}
\right].
\]
Hence, from the proof of Proposition 2.6 of \cite{CLW} we have 

\begin{corollary}
The quasi-Riordan array $[1+g_1t, f_1t+f_2t^2]$ is TP if and only if $g_1,f_1,f_2\geqslant 0$.
\end{corollary}

Proposition \ref{pro:Toeplitz_case} can be somewhat generalized to the following.

\begin{proposition}
\label{pro:necces_cond}
If $[g,f]$ is $TP$, then $f$ is a P\'olya frequency sequence. 
\end{proposition} 

Clearly, the above necessary condition is absolutely not sufficient for total positivity of $[g,f]$.  
Primarily, if the terms of $g$ are not nonnegative, then $[g,f]$ is definitely not TP, even if $f$ is a P\'olya frequency sequence. Moreover, it should also be emphasized again that, although the fact that both $g$ and $f$ are P\'olya frequency sequences implies that the Riordan array $(g,f)$ is TP \cite{CW}, the analogous result for quasi-Riordan arrays does not hold as we have seen in Example \ref{ex:gf_not_TP}. 

\section{Sequence Characterization of the TP of Quasi-Riordan Arrays}
\label{Sec3}

Let us now bring the attention to the fact that the quasi-Riordan array $[g,f]$ can be treated as almost-Riordan array $\left(g|\frac{f(t)}{t},t\right)$ \cite{B}. From \cite{AK}, we obtain the generating functions of the $A$-, $Z$-, and $W$-sequences of the latter one as follows:

\begin{align}
&A(t)=1, \label{eq:A_Z_W-1}\\
&Z(t)=\frac{f(t)-z_0tg(t)}{f(t)}+z_0, \label{eq:A_Z_W-2}\\ 
&W(t)=\frac{(1-w_0t)g(t)-1}{f(t)}+w_0,\label{eq:A_Z_W-3}
\end{align}
where $z_0=f_1$ and $w_0=g_1$. For instance, for quasi-Riordan array $[1/(1-t), t/(1-t)]$, $A(t)=Z(t)=W(t)=1$. 

We call the following matrix the {\it production matrix} of the quasi-Riordan matrix $[g,f]$:

\begin{equation}\label{eq:2}
J=\left[ \begin{array}{llllll} w_0& z_0& & & &  \\
w_1& z_1& a_0& & & \\ w_2& z_2& a_1& a_0 & & \\
w_3& z_3& a_2& a_1& a_0 &\\
\vdots &\vdots &\vdots &\vdots& \vdots &\ddots
\end{array}\right]
\end{equation}

\begin{proposition}\label{pro:3.0}
Let $[g,f]$ be a quasi-Riordan array, and let $J$ be the matrix defined by \eqref{eq:2}. Then, 

\begin{equation}\label{eq:3}
[g,f]J=\overline{[g,f]},
\end{equation}
where $\overline{[g,f]}$ is the matrix by deleting the first row of $[g,f]$.
\end{proposition}

\begin{proof}
Equation \eqref{eq:3} can be checked straightforwardly by using matrix multiplication and definition of quasi-Riordan array.
\end{proof}

\begin{theorem}\label{pro:general}
Let $J$ be a matrix defined by \eqref{eq:2} and let $R$ be the quasi-Riordan array defined by $A$-, $Z$-, and $W$-sequences. If $J$ is TP, then so is $R$. 
\end{theorem}

\begin{proof}
It can be proved by using \eqref{eq:3} and mathematical induction.
\end{proof}

\begin{corollary}\label{pro:2.1}
Let $[g,f]$ and $J$ be two matrices defined by \eqref{2.6} and \eqref{eq:2}, respectively. If $J$ is TP, then 
so are $[g,f]$ and $(g,f)$. 
\end{corollary}

\begin{proof}
The corollary can be proved by using Theorems \ref{thm:2.4} and \ref{pro:general}.
\end{proof}

Clearly, by (\ref{eq:A_Z_W-1})-\eqref{eq:A_Z_W-3} the matrix $J$ shown above is of a special form. Namely

\begin{equation}
\label{eq:J_for_[g,f]}
J=\left[ 
\begin{array}{lllllll}
w_0&z_0&&&&&\\
w_1&z_1&1&&&&\\
w_2&z_2&0&1&&&\\
w_3&z_3&0&0&1&&\\
w_4&z_4&0&0&0&1&\\
\vdots&\vdots&&&&&\ddots\\
\end{array}
\right].
\end{equation}

\noindent
Then, straight from the definition of total positivity one derives:

\begin{theorem}\label{prop:minors}
The matrix $J$ given by $(\ref{eq:J_for_[g,f]})$ is TP if and only if the two followings hold:
\begin{enumerate}
\item $w_k=z_k=0$ for all $k\geqslant 2$ and $w_0, w_1, z_0, z_1\geq 0$
\item $w_0z_1-w_1z_0\geqslant 0$.
\end{enumerate}
\end{theorem}

\begin{proof}
A matrix is TP if and only if its all minors are nonnegative.

Since minors of order $1$ must be nonnegative, we must have $w_k,z_k\geqslant 0$. 

If a minor does not consist neither the $0$-th nor the $1$-st row, then it is a minor of an identity matrix with an extra $0$ row on the top which is TP.

Now notice that choosing a minor consisting of the first and $k$-th  ($k\geqslant 2$) row and $0$-th and second column we get 
\[
\left|\begin{array}{ll}
w_1 & 1 \\
w_k & 0 \\
\end{array}\right|=-w_k\geqslant 0,
\] 
so comparing with the previous observation $w_k=0$ for all $k\geqslant 2$. 

Analogously for $z_k$.

Thus $J$ takes now the form 

\[
\left[ 
\begin{array}{lllllll}
w_0&z_0&&&&&\\
w_1&z_1&1&&&&\\
0&0&0&1&&&\\
0&0&0&0&1&&\\
0&0&0&0&0&1&\\
\vdots&\vdots&&&&&\ddots\\
\end{array}
\right]
\]

Consider the minors containing the $0$-th column and some columns $j_2$, $\ldots$, $j_n$, where $1<j_2<\cdots<j_n$. 
\begin{enumerate}
\item If a minor does not contain any of the rows $0$ or $1$, then its first column (coming after the $0$th one) is equal to $0$ and so is our minor. 
\item If a minor contains the $0$-th row, then it is triangular and equal either to $w_0$ or $0$. 
\item If a minor does not contain the $0$-th row, then exapanding it along this row we get that it is equal either to $0$ or $w_1$. 
\end{enumerate}

For the minors containing the first column and some columns $j_2$, $\ldots$, $j_n$, where $1<j_2<\cdots<j_n$, the conclusions are the same.

Consider now minors containing both columns $0$ and $1$. 
\begin{enumerate}
\item If a only the $0$-th row or the first row, then the first two columns are linearly dependent, so the minor is equal to $0$.
\item Suppose that a minor contains the $0$-th and the first row. If its order is $2$, then it is simply equal to 
$w_0z_1-w_1z_0$. Otherwise, it is a minor of a partitioned matrix of the form
\[
\left[\begin{array}{l|l}
A & B \\
\hline 
0 & C \\
\end{array}\right],
\quad\textrm{with } 
A=\left[\begin{array}{ll}w_0&z_0\\w_1&z_1\\\end{array}\right]
\]
and the first column of $C$ is zero, so the whole minor is equal to $0$ as well.
\end{enumerate}
\end{proof}

Joining the results of Theorems \ref{pro:general}, \ref{prop:minors} and $(\ref{eq:A_Z_W-1})$-$(\ref{eq:A_Z_W-3})$ one can derive the following.

\begin{corollary}
\label{cor:gf_TP}
Suppose $w_0,w_1,z_0,z_1\geqslant 0$ and $w_0z_1-w_1z_0\geqslant 0$. 
Then the quasi-Riordan array 
\[
\left[\frac{1-z_1t}{(w_0z_1-w_1z_0)t^2-(w_0+z_1)t+1},\frac{z_0t}{(w_0z_1-w_1z_0)t^2-(w_0+z_1)t+1}\right]
\] 
is TP. 
\end{corollary}

\begin{proof}
Substituting $A(t)=1$, $Z(t)=z_0+z_1t$ and $W(t)=w_0+w_1t$ into $(\ref{eq:A_Z_W-2})$ and \eqref{eq:A_Z_W-3} yields the system

\[
\begin{cases}
z_0tg(t)+(z_1t-1)f(t)=0\\
(w_0t-1)g(t)+w_1tf(t)=-1\\
\end{cases}
\]
which has the solution 
\[
\begin{array}{l}
g(t)=\frac{1-z_1t}{(w_0z_1-w_1z_0)t^2-(w_0+z_1)t+1},\\
f(t)=\frac{z_0t}{(w_0z_1-w_1z_0)t^2-(w_0+z_1)t+1}.\\
\end{array}
\]
By Theorem \ref{prop:minors} $J$ is TP. Thus, by Proposition \ref{pro:general} so is the obtained quasi-Riordan array.  
\end{proof}

\begin{example}
By Corollary \ref{cor:gf_TP} the quasi-Riordan array (with $w_0=z_0=1$, $w_1=2$, $z_1=3$):
\[
\left[\frac{1-3t}{t^2-4t+1},\frac{t}{t^2-4t+1}\right]
=\left[\begin{array}{llllll}
1&&&&&\\
1&1&&&&\\
3&4&1&&&\\
11&15&4&1&&\\
41&56&15&4&1&\\
\vdots&&&&&\ddots\\
\end{array}\right]
\]
is TP.
\end{example}

Note that since the discriminant of the denominator $D(t)$ of the obtained $g$ and $f$ is equal to 
\[
\Delta=w_0^2+z_1^2-2w_0z_1+4w_1z_0=(w_0-z_1)^2+4w_1z_0,
\]
all the coefficients are nonnegative, so $D(t)$ has two real roots. 
Moreover, according to the assumption of $w_0z_1-w_1z_0\geqslant 0$, their sum $\frac{w_0+z_1}{w_0z_1-w_1z_0}$ and their product $\frac 1{w_0z_1-w_1z_0}$ are nonnegative. Then, they must be nonnegative as well. 
Thus, $f$ is of form \eqref{0}, so by \cite{AESW, ASW, Kar} this confirms that $f$ must be P\'{o}lya frequency formal power series. If $z_1=0$, then by $w_0z_1-w_1z_0\geq 0$ we have $w_1z_0=0$, which implies $w_1=0$ due to $z_0\not= 0$. Hence, if $z_1=0$, then $z_0> 0$, $w_1=0$, and $w_0> 0$, and all those make $[g,f]=\left[ \frac{1}{1-w_0t}, \frac{z_0t}{1-w_0t}\right]$, where both $f$ and $g$ are P\'olya frequency. If  $z_1>0$, then $g$ is not a P\'{o}lya frequency formal power series, although both $[g,f]$ and, hence, $(g,f)$ are TP. This gives another counterexample for Item 2 in \cite{Slo}. As what we have seen, if  $z_1>0$, then $g$ is not a P\'{o}lya frequency formal power series-- its numerator has a positive root.  However, as Proposition \ref{pro:general} holds, $[g(t),f(t)]$ is TP, so the coefficients of the series $g(t)$ are nonnegative. 
Moreover, we have already seen in Example \ref{ex:gf_not_TP} that, in general, fact that both series $g(t)$ and $f(t)$ are P\'olya frequency  is not sufficient for $[g,f]$ to be TP, so one could mistrust the necessity of this condition as well.

\section{Counterexamples}\label{Sec4}

Note that in Example \ref{ex:gf_not_TP} we have already presented a non-TP quasi-Riordan array $[g,f]$, where $g$ and $f$ are P\'olya frequency power series. Here, in the first subsection, we give some more examples like that.  
In the second subsection, we give some more examples for TP of quasi-Riordan arrays $[g,f]$ with non-P\'olya frequency power series $g$.

\subsection{Quasi-Riordan arrays that are not TP}\label{Subsec4.1}

\begin{proposition}
\label{pro:alpha}
Let $k_1,k_2$ be natural numbers such that $k_2>k_1$, $\alpha>0$, and let $f(t)=\sum_{k\geqslant 1}f_kt^k$ be a P\'{o}lya frequency power series. If for some positive natural number $n$ it holds $\alpha^{k_2-k_1}f_{k_1-n+1}>f_{k_2-n+1}$, then the quasi-Riordan array $\left[\frac 1{1-\alpha t},f(t)\right]$ is not TP, although both series $g$ and $f$ are P\'{o}lya frequency power series.
\end{proposition}

\begin{proof}
Consider a minor of $\left[\frac 1{1-\alpha t},f(t)\right]$ that consists of the entries from the rows $k_1$, $k_2$ and the columns $0$ and $n$. It is equal to 
\[
M_{0,n}^{k_1,k_2}=\alpha^{k_1}f_{k_2-n+1}-\alpha^{k_2}f_{k_1-n+1},
\]
so by the assumption it is negative, and so $\left[\frac 1{1-\alpha t},f(t)\right]$ is not TP. 
\end{proof}

\begin{example}
Let 
\[
f(t)=\frac t{(1-2t)^2}=t+4t^2+12t^3+32t^4+80t^5+192t^6+\cdots
\]
Choosing $k_1=3$, $k_2=4$, $n=1$, for $\alpha=3$ we have 
\[
\alpha^{4-3}f_{3}>f_{4}.
\]
Hence, 
\[
\left[\frac 1{1-3t},\frac t{(1-2t)^2}\right]=\left[
\begin{array}{llllll}
1&&&&&\\
3&1&&&&\\
9&4&1&&&\\
27&12&4&1&&\\
81&32&12&4&1&\\
\vdots&&&&&\ddots\\
\end{array}
\right]
\]
is not TP, as 
\[
M^{0,1}_{3,4}=\left|\begin{array}{ll}
27 & 12 \\
81 & 32 \\
\end{array}\right|=-108<0, 
\]
although $g$ and $f$ are P\'{o}lya frequency power series.
\end{example}

In general, Proposition \ref{pro:alpha} implies the following.

\begin{corollary}
For every P\'{o}lya frequency power series $f$, there exists a constant $\alpha>0$ such that the quasi-Riordan array $\left[\frac 1{1-\alpha t},f(t)\right]$ is not TP, although both $\frac 1{1-\alpha t}$ and $f(t)$ are P\'{o}lya frequency power series.
\end{corollary}

\begin{proof}
Clearly, fixing arbitrarily indices $k_1,k_2,n$, $k_1<k_2$, $n>0$, it suffices to take 
\begin{equation}
\label{eq:alpha} 
\alpha>\sqrt[k_2-k_1]{\frac{f_{k_2-n}}{f_{k_1-n}}}.
\end{equation}  
\end{proof}

\begin{example}
Let 
\[
f(t)=\frac{t(t+1)}{1-2t}=t+\sum_{n=2}^\infty\left(2^{n-2}+2^{n-1}\right)t^n.
\]
To find $\alpha$ appearing in Proposition \ref{pro:alpha} we consider the inequality 
\[
\alpha>\sqrt[k_2-k_1]{\frac{f_{k_1-n}}{f_{k_2-n}}}.
\]
Taking, for instance $k_2=n+2$, $k_1=n+1$, we get 

\[
\alpha>\frac{f_2}{f_1}=3.
\]
Thus, for any $\alpha>3$ the quasi-Riordan $\left[\frac 1{1-\alpha t},f(t)\right]$ is not TP.
\end{example}

Series $g$, that in from Proposition \ref{pro:alpha} is equal to $\frac 1{1-\alpha t}$, can be generalized to various other forms. A formula for this general form seems difficult to be found, yet let us present one more example.

\begin{example}
\label{ex:alpha_beta}
Let $f(t)$ be a P\'{o}lya frequency power series and let 
\[
g(t)=\frac{1}{(1-\alpha t)(1-\beta t)}=\sum_{n=0}^\infty\frac{\beta^{n+1}-\alpha^{n+1}}{\beta-\alpha}t^n\quad\textrm{with }\beta>\alpha,
\]
i.e. $g(t)$ is also a P\'{o}lya frequency power series. 
\par 
Looking at the minor 
\[
M^{1,2}_{0,1}=\left|\begin{array}{ll}
\frac{\beta^2-\alpha^2}{\beta-\alpha} & f_0 \\
\frac{\beta^3-\alpha^3}{\beta-\alpha} & f_1 \\
\end{array}\right|
\]
it can be observed that it is negative if and only if 
\[
\frac{\alpha^2+\alpha\beta+\beta^2}{\alpha+\beta}>\frac{f_1}{f_0}
\quad\Leftrightarrow\quad 
\alpha^2+\beta^2+\alpha\beta-\frac{f_1}{f_0}\alpha-\frac{f_1}{f_0}\beta>0.
\]
The above inequality has infinitely many solutions. One of the examples can be the pair $\alpha=\frac{f_1}{2f_0}$, $\beta=\frac{f_1}{f_0}$. 
More precisely, the set of all solutions is the exterior of the ellipse 
\[
\frac{\left(x-\frac{\sqrt 2f_1}{3f_0}\right)^2}{\frac{2f_1^2}{9f_0^2}}+\frac{y^2}{\frac{2f_1^2}{3f_0^2}}=1
\]
rotated by the angle $-\frac{\pi}4$. The example of this area for $\frac{f_1}{f_0}=2$ is shown on Fig.\ref{fig:ellipse}. 
\end{example}

\begin{figure}[p!]
\caption{The area depicting pairs $(\alpha,\beta)$ for which $\left[\frac 1{(1-\alpha t)(1-\beta t)},f(t)\right]$ is not TP, $\frac{f_1}{f_0}=2$. Note that the fact that point $(\alpha,\beta)$ lies inside the ellipse does not necessarily mean that $[g,f]$ is TP.}
\label{fig:ellipse}
\begin{center}
\includegraphics[scale=0.4]{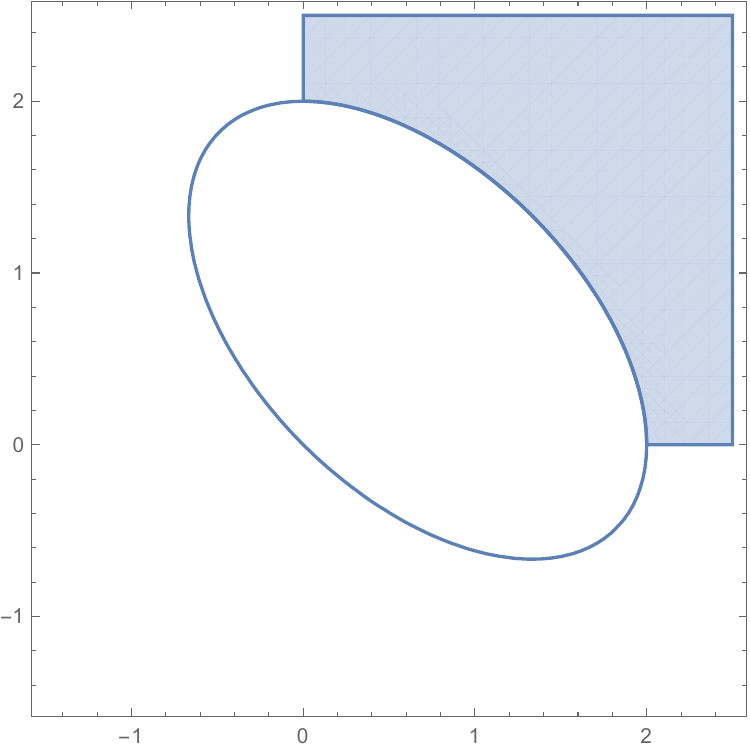}
\end{center}
\end{figure}

\subsection{Quasi-Riordan arrays that are TP}\label{Subsec4.2}

\begin{example}
\label{ex:g_not}
Let $f(t)=\frac t{1-2t}$, $g(t)=1+t+t^2$. Then $f$ is a P\'{o}lya frequency power series, $g$ is not. 
Consider 
\[
R=[g,f]=\left[\begin{array}{llllll}
1&&&&&\\
1&1&&&&\\
1&2&1&&&\\
0&4&2&1&&\\
0&8&4&2&1&\\
\vdots&&&&&\ddots\\
\end{array}\right]
\]
and its minors. 

Clearly, those minors that do not involve the $0$th column are the minors of $\left[1,f(t)\right]$, so they are nonnegative. Moreover, only those that are either of order $1$ or those that are triangular minors are positive, the other ones are equal to $0$. 

Consider now a minor of order $n$ ($n\geqslant 2$) containing the $0$th column. If it contains the $0$th row, then expanding along the $0$th row, we get that it is equal to some minor of $\left[1,f(t)\right]$, so it is nonnegative. Suppose that it does not not contain the $0$th row. Then let us expand it along the $0$th column. If it contains one of the $1$st row and the $2$ row, expanding along the $0$th column, we again get a minor of $\left[1,f(t)\right]$. Suppose that it contains both $1$st and $2$nd row. Again, expanding along the $0$th column, we get then the difference of two minors of $\left[1,f(t)\right]$. If $n=2$, then we deal with the minor 
\[
\left|\begin{array}{ll}
1 & r_{k\:1} \\
1 & r_{k\:2} \\
\end{array}\right|\in 
\left\{
\left|\begin{array}{ll}
1 & 1 \\
1 & 2 \\
\end{array}\right|,\; 
\left|\begin{array}{ll}
1 & 0 \\
1 & 1 \\
\end{array}\right|,\;
\left|\begin{array}{ll}
1 & 0 \\
1 & 0 \\
\end{array}\right|
\right\}
\]
that is nonnegative. 

If $n\geqslant 3$, both minors are of order greater or equal to $3$, so they are equal to $0$, and so is their difference.  Hence, $R$ is TP. 
\end{example}

Example \ref{ex:g_not} can be extended to the following.

\begin{corollary}
\label{cor:g_not}
Let $\alpha>0$, and let $g_0,g_1,g_2>0$ be such that $g_1^2-4g_0g_2<0$. 
If $g_1\alpha-g_2\geqslant 0$, then $R=\left[g_0+g_1t+g_2t^2,\frac t{1-\alpha t}\right]$ is TP, although $g_0+g_1t+g_2t^2$ is not a P\'{o}lya frequency power series. 
\end{corollary}

\begin{proof}
Analogously as in the example, the only minors that are nonzero, are the triangular ones and the minor 
\[
\left|\begin{array}{ll}
r_{10} & r_{11} \\
r_{20} & r_{21} \\
\end{array}
\right|=\left|\begin{array}{ll}
g_1 & 1 \\
g_2 & \alpha \\
\end{array}\right|=g_1\alpha-g_2
\]
which, by assumption, is nonnegative.
\end{proof}

Note that replacing $g(t)=1+t+t^2$ from Example \ref{ex:g_not} by, for instance, $g(t)=t^2+1$ (that does not satisfy the assumption $g_1>0$ of Corollary \ref{cor:g_not}) yields the array of the form 
\[
\left[\begin{array}{lllll}
1&&&&\\
0&f_1&&&\\
1&f_2&f_1&&\\
0&f_3&f_2&f_1&\\
\vdots&&&&\ddots\\
\end{array}\right]
\]
whose minor $M^{1,2}_{0,1}$ is equal to $-f_1<0$, so the above array is not TP.

\section{Closing comments}

In the present work we have studied total positivity of quasi-Riordan arrays. As we have mentioned, these are, in fact, some particular cases of almost-Riordan arrays. Thus, it is natural question whether the presented results can be generalized for almost-Riordan arrays. Considering the capacity of the paper, the TP of almost-Riordan arrays will be discussed in a subsequent paper. For the same reason, some other related topics, such as analogues of $TPr$ Riordan arrays in \cite{CLW}, $TP_r$ quasi-Riordan arrays and $TPr$ almost-Riordan arrays, will also be discussed in a subsequent paper.

\section{Data availability statement}

All data generated or analyzed during this study are included in this article.


\begin{thebibliography}{99}
\bibitem{AESW}
M. Aissen, A. Edrei, I.J. Schoenberg and A.M. Whitney, On the generating functions of totally positive sequences, {\it Proc. Nat. Acad. Sci. U.S.A.} 37 (1951), 303--307.

\bibitem{ASW}  
M. Aissen, I.J. Schoenberg and A.M. Whitney, On the generating functions of totally positive sequences I, {\it J. Analyse Math.} 2 (1952), 93--103.
 
\bibitem{AK}
Y. Alp and E. G. Kocer, Sequence characterization of almost-Riordan arrays, {\it Linear Algebra Appl.} 664 (2023), 1--23.

\bibitem{Barry}
P. Barry, \textit{Riordan Arrays: A Primer}, LOGIC Press, Kilcock, 2016.

\bibitem{B} 
P. Barry, On the group of almost-Riordan arrays, arXiv:1606.05077v1.

\bibitem{Bre}
F. Brenti, Combinatorics and total positivity, {\it J. Combin. Theory Ser. A}, 71 (1995) 175--218.

\bibitem{CLW} 
X. Chen, H. Liang and Y. Wang, Total positivity of Riordan arrays, {\it European J. Combin.} 46 (2015) 68--74.

\bibitem{CW}
X. Chen and Y. Wang, Notes on the total positivity of Riordan arrays, {\it Linear Algebra Appl.} 569 (2019) 156--161.

\bibitem{He}
T.-X. He, The vertical recursive relation of Riordan arrays and its matrix representation, {\it J. Integer Seq.} 25 (2022), no. 9, Art. 22.9.5, 22 pp.

\bibitem{He15} 
T. -X. He, Matrix characterizations of Riordan matrices, {\it Linear Algebra Appl.} 465 (2015), 15-42.

\bibitem{HS}
T. -X. He and R. Sprugnoli. Sequence Characterization of Riordan Arrays, {\em Discrete Math.}, 309 (2009), 3962-3974.

\bibitem{JLN}
C. Jean-Louis and A. Nkwanta, Some algebraic structure of the Riordan group, {\it Linear Algebra Appl.} 438 (2013), 2018--2035.

\bibitem{Kar}
S. Karlin, {\it Total Positivity, Vol.1}, Stanford University Press, 1968.

\bibitem{LMMPMS}
A. Luz\'o n, D. Merlini, M. A. Mor\'o n, L. F. Prieto-Martinez, and R. Sprugnoli, Some inverse limit approaches to the Riordan group, {\it Linear Algebra Appl.} 491 (2016), 239--262.

\bibitem{LMPM17} 
A. Luz\'o n, M. A. Mor\'o n, and L. F. Prieto-Martinez, 
A formula to construct all involutions in Riordan matrix groups. Linear Algebra Appl. 533 (2017), 397--417. 

\bibitem{LMPM}
A. Luz\'o n, M. A. Mor\'o n, and L. F. Prieto-Martinez,  The group generated by Riordan involutions, {\it Revista Matem\' atica Complutense} https://doi.org/10.1007/s13163-020-00382-8, 2020. 

\bibitem{MMW}
J. Mao, L. Mu, and Y. Wang, Yet another criterion for the total positivity of Riordan arrays, {\it  Linear Algebra Appl.} {\bf 634} (2022), 106--111. 

\bibitem{MRSV}
D. Merlini, D. G. Rogers, R. Sprugnoli, and M. C. Verri, On some alternative characterizations of Riordan matrices, {\em Canadian J. Math.}, 49 (1997), 301--320.

\bibitem{Pin}
A. Pinkus, {\it Totally Positive Matrices}, Cambridge University Press, Cambridge, 2010.

\bibitem{Rog}
D. G. Rogers, Pascal triangles, Catalan numbers and renewal matrices, {\em Discrete Math.}, 22 (1978), 301--310.

\bibitem{Sha}
L. W. Shapiro, Bijections and the Riordan group, Random generation of combinatorial objects and bijective combinatorics,  {\em Theoret. Comput. Sci.} 307 (2003), no. 2, 403--413.

\bibitem{SGWW}
L. W. Shapiro, S. Getu, W. J. Woan and L. Woodson,  The Riordan group, {\em Discrete Appl. Math.} 34(1991) 229--239.

\bibitem{SSBCHMW}
L. W. Shapiro, R. Sprugnoli, P. Barry, G.-S. Cheon, T.-X. He, D. Merlini, and W. Wang,  {\it The Riordan Group and Applications}, Springer, 2022. 

\bibitem{Slo}
R. Slowik, Some (counter)examples on totally positive Riordan arrays, {\it Linear Algebra Appl.} 594 (2020), 117--123.
\end{thebibliography}
\end{document}